\documentclass[12pt]{amsart}

\usepackage{verbatim}
\usepackage{fullpage}
\usepackage{amsfonts}

\usepackage{graphicx}
\usepackage{epstopdf}
\usepackage{amsmath, amsthm, amssymb}
\usepackage{fancybox}
\newcommand{\swoosh}{\includegraphics[width=0.35in]{swoosh.pdf}}

\newcommand{\R}{\mathbb{R}}
\newcommand{\C}{\mathbb{C}}
\newcommand{\N}{\mathbb{N}}
\newcommand{\Z}{\mathbb{Z}}

\newcommand{\Hil}{\mathcal{H}}

\DeclareMathOperator{\ran}{ran}

\newcounter{Theorem}

\numberwithin{equation}{section}
\numberwithin{Theorem}{section}

\theoremstyle{plain} 
\newtheorem{thm}[Theorem]{Theorem}

\theoremstyle{definition}
\newtheorem{defn}[Theorem]{Definition}

\theoremstyle{remark}
\newtheorem{remark}{Remark}[section]
\newtheorem{ex}[Theorem]{Example}

\begin{document}

\title{Diagonals of self-adjoint operators with finite spectrum}

\author{Marcin Bownik}
\address{Department of Mathematics, University of Oregon, Eugene, OR 97403--1222, USA}
\email{mbownik@uoregon.edu}

\author{John Jasper}
\address{Department of Mathematics,
University of Missouri,
Columbia, MO 65211--4100, USA}
\email{jasperj@missouri.edu}


\thanks{
The first author was partially supported by NSF grant DMS-1265711 and by the Simons Foundation grant \#244422.
The second author was supported by NSF ATD 1042701}

\date{\today}

\begin{abstract}
 Given a finite set $X\subseteq\R$ we characterize the diagonals of self-adjoint operators with spectrum $X$. Our result extends the Schur-Horn theorem from a finite dimensional setting to an infinite dimensional Hilbert space analogous to Kadison's theorem for orthogonal projections \cite{k1,k2} and the second author's result for operators with three point spectrum \cite{jj}. 
\end{abstract}

\maketitle


\section{Introduction}


The characterization of the diagonals of self-adjoint operators began with the classical Schur-Horn theorem \cite{horn, schur}. It can be stated as follows, where $\mathcal H_N$ is $N$ dimensional Hilbert space over $\R$ or $\C$, i.e., $\mathcal H_N=\R^N$ or $\C^N$.

\begin{thm}[Schur-Horn theorem]\label{horn} 
Let $\{\lambda_{i}\}_{i=1}^{N}$ and $\{d_{i}\}_{i=1}^{N}$ be real sequences in nonincreasing order. There exists a self-adjoint operator $E:\mathcal H_N \to\mathcal H_N$ with eigenvalues $\{\lambda_{i}\}$ and diagonal $\{d_{i}\}$
if and only of 
\begin{equation}\label{horn1}
\sum_{i=1}^{N}\lambda_{i} =\sum_{i=1}^{N}d_{i} \quad\text{ and }\quad \sum_{i=1}^{n}d_{i} \leq \sum_{i=1}^{n}\lambda_{i} \text{ for all } 1\le n \le N.
\end{equation}
\end{thm}

The necessity of \eqref{horn1} is due to Schur \cite{schur}, and the sufficiency of \eqref{horn1} is due to Horn \cite{horn}. 
There are two problems that immediately arise in extending the Schur-Horn Theorem \ref{horn} to operators on infinite dimensional Hilbert spaces. First, there is no obvious analogue of the eigenvalue sequence $\{\lambda_{i}\}$ for nondiagonalizable operators. Second, even if the operator in question is diagonalizable, the diagonal sequence and eigenvalue sequence cannot generally be rearranged in nonincreasing order.

Positive compact operators are diagonalizable, and both the eigenvalue sequence and diagonal are $c_{0}$ sequences. Thus, a natural analogue of the Schur-Horn theorem can be formulated in terms of majorization inequalities as in \eqref{horn1}. Results of this kind were proven by Gohberg and Markus \cite{gm} and Arveson and Kadison \cite{ak} for trace class operators and extended to positive compact operators by Kaftal and Weiss \cite{kw}. Though the theorem statements are a natural analogue of the Schur-Horn theorem, the proof are much more intricate. Moreover, a complete characterization of the diagonals of compact operators with finite dimensional kernels is not yet complete. For a detailed survey of progress in this area see \cite{kw0}.

It is well known that \eqref{horn1} can be stated by the equivalent convexity condition
\begin{equation}\label{horn2}
(d_1,\ldots,d_N) \in \operatorname{conv} \{ (\lambda_{\sigma(1)},\ldots,\lambda_{\sigma(N)}): \sigma \in S_N\},
\end{equation}
where $S_N$ is a permutation group on $N$ elements. This reformulation does not require an ordering of sequences. In \cite{neu} Neumann gave an infinite dimensional generalization of the Schur-Horn theorem in terms of $\ell^\infty$-closure of the convexity condition \eqref{horn2}. This gives a nice analogue of the finite dimensional theorem, but a great deal of information is lost in taking the closures. Indeed, consider an orthogonal projection $P$ with infinite dimensional kernel and range. From Neumann's work it can only be deduced that the closure of the set of diagonals of $P$ is the set of sequences $\{d_{i}\}$ in $[0,1]$ such that $\sum d_{i} = \sum (1-d_{i})=\infty$. 

A complete characterization of the diagonal sequences of orthogonal projections was discovered by Kadison \cite{k1,k2}. By Theorem \ref{horn}, the necessary and sufficient condition for a finite sequence $\{d_{i}\}_{i=1}^{N}$ in $[0,1]$ to be the diagonal of an orthogonal projection on $\mathcal H_N$ is simply the trace condition $\sum d_{i} \in\N_{0}$. Kadison's theorem gives an elegant analogue of the trace condition on an infinite Hilbert dimensional space $\mathcal H$, see Theorem \ref{Kadison}(ii). Crucially, it does not require an ordering of the terms of the diagonal sequence.
By scaling and translating Kadison's theorem gives a characterization of all self-adjoint operators with two points in the spectrum. Thus, a natural next step after Kadison's work is to look at the diagonals of self-adjoint operators with finite spectrum. As with orthogonal projections these operators are diagonalizable, but we still face the problem that the diagonal terms cannot be ordered. We shall overcome this problem by introducing the concept of interior majorization in Definition \ref{mar} which does not require any ordering.

An additional issue arises that is not present in the case of projections. Kadison's Theorem \ref{Kadison} characterizes the diagonals of the set of \textbf{all} projections. Despite this, the multiplicities of the eigenvalues of a projection $P$ can be easily deduced from the diagonal sequence $\{d_{i}\}$ by
\[\dim\ran P = \sum d_{i}\quad\text{and}\quad\dim\ker P = \sum(1-d_{i}).\]
Thus, Kadison's theorem yields a complete characterization of diagonal sequences of the unitary orbit of a projection $P$. In contrast, there exist operators with the same finite spectrum and the same diagonal that are not unitarily equivalent, see Example \ref{nu}. In other words, self-adjoint operators having different unitary orbits may share the same diagonal.

This leads us to two distinct versions of the Schur-Horn theorem for operators with finite spectrum, which were already present in the second author's work \cite{jj} on operators with 3 point spectrum. In this paper we shall focus on a spectral version of the theorem. That is, we give a characterization of the diagonals of the set of all self-adjoint operators with a given spectrum. In the follow-up paper \cite{mbjj3} we give a version of this result with prescribed multiplicities of eigenvalues.

Our main result can be thought as an analogue of the work by Arveson \cite{a} who identified some necessary conditions which must be satisfied by diagonals of normal operators with finite spectrum. Unlike \cite{a} our main result deals only with self-adjoint operators. On the other hand, Theorem \ref{npt} gives a complete characterization of diagonals of self-adjoint operators with finite spectrum.

\begin{thm}\label{npt} Let $\{A_j\}_{j=0}^{n+1}$ be an increasing sequence of real numbers such that $A_0=0$ and $A_{n+1}=B$, $n\in \N$. Let $\{d_{i}\}_{i\in I}$ be a sequence in $[0,B]$ with $\sum d_{i} = \sum (B-d_{i}) = \infty$. For each $\alpha \in (0,B)$, define
\begin{equation}\label{conv}
C(\alpha) = \sum_{d_{i}<\alpha}d_{i}\quad\text{and}\quad D(\alpha)=\sum_{d_{i}\geq \alpha}(B-d_{i}).
\end{equation}
There exists a self-adjoint operator $E$ with diagonal $\{d_{i}\}_{i\in I}$ and $\sigma(E) = \{A_0, A_1, \ldots, A_{n+1}\}$ if and only if  either:
\begin{enumerate}
\item $C(B/2)=\infty$ or $D(B/2)=\infty$, or 
\item $C(B/2)<\infty$ and $D(B/2)<\infty$, $($and thus $C(\alpha), D(\alpha)< \infty$ for all $\alpha \in (0,B)$$)$, and there exist $N_1, \ldots, N_n \in\N$ and $k\in\Z$ such that:
\begin{equation}\label{npttrace} 
C(B/2)-D(B/2) = \sum_{j=1}^n A_j N_j+kB,
\end{equation}
and for all $r=1,\ldots, n$,
\begin{equation}
\label{nptmaj} (B-A_r)C(A_r)+A_r D(A_r) \geq (B-A_r)\sum_{j=1}^r A_j N_j +  A_r\sum_{j=r+1}^n (B-A_j)N_j.
\end{equation}
\end{enumerate}
\end{thm}

We remark that the assumption that $\sum d_{i} = \sum (B-d_{i}) = \infty$ is not a true limitation of Theorem \ref{npt}. Indeed, the summable case $\sum d_{i}<\infty$, or its symmetric variant  $\sum (B-d_{i})<\infty$, leads to a finite rank Schur-Horn theorem, see \cite[Theorems 2.1 and 2.2]{mbjj3}. This case requires a different set of conditions which are closely related to the classical Schur-Horn majorization. Finally, the assumption $A_0=0$ is made only for simplicity; the general case follows immediately by a translation argument.

We should also emphasize that the numbers $N_1,\ldots,N_n$ in general do  not correspond to multiplicities of eigenvalues $A_1,\ldots,A_n$, see Example \ref{nu}. This is unlike the main theorem in \cite{mbjj3}, where the numbers $N_1,\ldots,N_n$ represent multiplicities of eigenvalues in the interior majorization inequality \eqref{nptmaj}. In addition, the main result of \cite{mbjj3} has much more complicated statement since it also involves exterior majorization conditions that are very sensitive to the locations of eigenvalues with infinite multiplicities. Hence, the multiplicity-free Theorem \ref{npt}, though theoretically deductible from its counterpart, is not an easy consequence of \cite{mbjj3}. 
For these reasons we shall give a direct proof of Theorem \ref{npt}. The additional benefit of our approach is that it gives a short and largely independent proof, which does not rely on a long argument showing \cite[Theorem 1.3]{mbjj3} in its entirety. Instead, we merely use two particular results from \cite{mbjj3} on the equivalence of Riemann and Lebesgue interior majorization and the sufficiency of Riemann majorization. Combining this with the key  existence result in the non-summable case from \cite{jj} yields the sufficiency part of Theorem \ref{npt}. The proof of the necessity part is self-contained as it relies only on Kadison's Theorem \ref{Kadison}.

\section{Necessity of interior majorization}\label{S4}

In this section we will show the necessity in Theorem \ref{npt}. We will make an extensive use of Kadison's theorem \cite{k1,k2} which characterizes diagonals of orthogonal projections. Theorem \ref{Kadison} serves as a prototype for our Theorem \ref{npt}. The common feature of both of these results is a trace condition. The main distinction between them is the lack of majorization inequalities \eqref{nptmaj} in Kadison's Theorem which are present in Theorem \ref{npt}.

\begin{thm}[Kadison]
\label{Kadison} Let $\{d_{i}\}_{i\in I}$ be a sequence in $[0,1]$ and $\alpha\in(0,1)$. Define
\[
C(\alpha)=\sum_{d_{i}<\alpha}d_{i}, \qquad D(\alpha)=\sum_{d_{i}\geq \alpha}(1-d_{i}).\]
There exists an orthogonal projection on $\ell^2(I)$ with diagonal $\{d_{i}\}_{i\in I}$ if and only if either:
\begin{enumerate}
\item $C(\alpha)=\infty$ or $D(\alpha)=\infty$, or
\item $C(\alpha)<\infty$ and $D(\alpha)<\infty$, and
\begin{equation}
\label{kadcond} 
C(\alpha)-D(\alpha)\in\Z.
\end{equation}
\end{enumerate}
\end{thm}

\begin{remark}\label{rpart}
Note that if there exists a partition of $I=I_{1}\cup I_{2}$ such that
\begin{equation}\label{partition}
\sum_{i\in I_{1}}d_{i}<\infty \quad\text{and}\quad\sum_{i\in I_{2}}(1-d_{i})<\infty,
\end{equation}
then for all $\alpha\in (0,1)$ we have $C(\alpha)<\infty$ and $D(\alpha)<\infty$ and
\[
\bigg(\sum_{i\in I_{1}} d_{i} - \sum_{i\in I_{2}}(1-d_{i}) \bigg) - (C(\alpha)-D(\alpha))\in\Z.
\]
Thus, in the presence of a partition satisfying \eqref{partition},
\[
\bigg(\sum_{i\in I_{1}} d_{i} - \sum_{i\in I_{2}}(1-d_{i}) \bigg) \in\Z
\]
is a necessary and sufficient condition for a sequence to the be the diagonal of a projection. We will find use for these more general partitions in the sequel.
\end{remark}

It is convenient to formalize the concept of interior majorization with the following definition.

\begin{defn}\label{mar}
Let $\{A_j\}_{j=0}^{n+1}$ be an increasing sequence such that $A_0=0$ and $A_{n+1}=B$, $n\in \N$. 
Let $\{d_i\}_{i\in I}$ be a sequence in $[0,B]$. Let $C(\alpha)$ and $D(\alpha)$ be as in \eqref{conv}. 

We say that $\{d_i\}$ satisfies {\it interior majorization} by $\{ A_j\}_{j=0}^{n+1}$ if the following 3 conditions hold:
\begin{enumerate}
\item
$C(B/2)<\infty$ and $D(B/2)<\infty$, and thus $C(\alpha)<\infty$ and $D(\alpha)< \infty$ for all $\alpha \in (0,B)$,
\item
there exist $N_1, \ldots, N_n \in\N$ and $k_0\in \Z$ such that
\begin{equation}
\label{fulltrace1} 
C(A_n) -D(A_n) =  \sum_{j = 1}^n A_j N_j +k_0 B,
\end{equation}
\item 
for all $r=1,\ldots,n$,
\begin{equation}
\label{fullmaj1} 
C(A_r) \ge\sum_{j=1}^{r} A_j N_j + A_r \bigg( k_0 - |\{i\in I: A_r \le d_i < A_n \}|+\sum_{j=r+1}^{n} N_j  \bigg).
\end{equation}
\end{enumerate}
\end{defn}

\begin{remark}
Despite its initial appearance, the interior majorization conditions \eqref{fulltrace1} and \eqref{fullmaj1} are equivalent with  \eqref{npttrace} and \eqref{nptmaj} in Theorem \ref{npt}. Indeed, by Remark \ref{rpart}, \eqref{fulltrace1} is equivalent to the statement that for all $\alpha \in (0,B)$ there exists $k=k(\alpha) \in \Z$ such that
\begin{equation}
\label{fulltrace2} 
C(\alpha) -D(\alpha) =  \sum_{j=1}^{n} A_j N_j +k(\alpha) B,
\end{equation}
Fix $\alpha = A_r$, where $r=1, \ldots, n$. Then, \eqref{fullmaj1} can be rewritten as 
\begin{equation}
\label{fullmaj2} 
C(\alpha) \ge\sum_{j =1}^{r} A_j N_j + \alpha \bigg( k(\alpha)+\sum_{j=r+1}^{n} N_j  \bigg).
\end{equation}
Using \eqref{fulltrace2}, we can remove the presence of $k=k(\alpha)$ in \eqref{fullmaj2} to obtain
\begin{equation}
\label{fullmaj3} 
(B-\alpha)C(\alpha) +\alpha D(\alpha)  \ge
(B-\alpha)\sum_{j = 1}^{r} A_j N_j + \alpha \sum_{j=r+1}^{n} (B-A_j)N_j.
\end{equation}
This is precisely \eqref{nptmaj}, and the above process is reversible. 
\end{remark}


\begin{thm}\label{int-nec} Let $E$ be a self-adjoint operator on $\mathcal H$ with spectrum 
\[
\sigma(E)=\{A_0,\ldots,A_{n+1}\},
\]
where $\{A_j\}_{j=0}^{n+1}$ is an increasing sequence such that $A_0=0$ and $A_{n+1}=B$, $n\in \N$. Let $d_{i}=\langle Ee_{i},e_{i}\rangle$ be a diagonal of $E$ with respect to some orthonormal basis $\{e_{i}\}_{i\in I}$ of $\Hil$. 
Assume that for some $0< \alpha<B$, 
$C(\alpha)<\infty$ and $D(\alpha)<\infty$.
Then,
$\{d_i\}_{i\in I}$ satisfies interior majorization by $\{A_j\}_{j=0}^{n+1}$.
\end{thm}

\begin{proof}
By the spectral decomposition, we can write
\[
E= \sum_{j=0}^{n+1} A_{j} P_j,
\]
where $P_j$'s are mutually orthogonal projections satisfying $\sum_{j=0}^{n+1}P_j={\mathbf I}$. Let $p^{(j)}_{i}=\langle P_je_{i},e_{i}\rangle$ be the diagonal of $P_j$. Hence, we have
\begin{equation}\label{tt10}
\sum_{j=0}^{n+1} p_i^{(j)}=1
\qquad\text{for all } i\in I.
\end{equation} 

For convenience let $I_0=\{i\in I: d_{i}<\alpha\}$ and $I_1=\{i\in I: d_{i} \ge \alpha\}$. By our assumption
\[
C(\alpha) = \sum_{i\in I_0} d_i <\infty, \qquad  D(\alpha) = \sum_{i\in I_1} (B- d_i)<\infty.
\]
Summing $d_i = \sum_{j=0}^{n+1} A_j p_i^{(j)}$ over $i\in I_0$ yields
\begin{equation}\label{tt11}
\sum_{i\in I_0} p_i^{(j)} < \infty \qquad\text{for }j=1,\ldots,n+1.
\end{equation}
Using \eqref{tt10} we have that $B-d_i = \sum_{j=0}^{n+1} (B-A_j) p_i^{(j)}$. Summing this over $i\in I_1$ yields
\begin{equation}\label{tt12}
\sum_{i\in I_1} p_i^{(j)} < \infty \quad\text{for }j=0,\ldots,n.
\end{equation}
Combining \eqref{tt11} and \eqref{tt12} and applying Theorem \ref{Kadison} yields
\begin{equation}\label{tt13}
N_j : = \sum_{i \in I} p_i^{(j)} \in\N \qquad j=1,\ldots,n.
\end{equation}
Moreover, since $ B(1-p_i^{(n+1)}) = B-d_i+ \sum_{j=0}^n A_j p_i^{(j)}$ by \eqref{tt12} we have
\begin{equation}\label{tt14}
\sum_{i\in I_1} (1-p_i^{(n+1)}) < \infty.
\end{equation}
By \eqref{tt11}, \eqref{tt14}, Theorem \ref{Kadison} and Remark \ref{rpart} applied to the projection $P_{n+1}$ we have
\[
k_0:=\sum_{i\in I_0} p_i^{(n+1)}-\sum_{i\in I_1} (1-p_i^{(n+1)}) \in \Z.
\]
Thus,
\[
\begin{aligned}
C(\alpha) - D(\alpha) &= \sum_{i\in I_0}  \bigg(Bp_i^{(n+1)} +  \sum_{j=1}^{n} A_j p_i^{(j)}  \bigg) - \sum_{i\in I_1} \bigg( B- Bp_i^{(n+1)}  - \sum_{j=1}^{n} A_j p_i^{(j)} \bigg) 
\\
&= Bk_0 + \sum_{j=1}^n N_j A_j.
\end{aligned}
\]
For convenience we let $q_i=p_i^{(n+1)}$. In particular, by letting $\alpha=A_n$ the above shows \eqref{fulltrace1} with $k_0=a-b$, where
\[
a:=\sum_{d_i<A_n}q_{i}<\infty,\qquad b:=\sum_{d_i \ge A_n}(1-q_{i})<\infty.
\]
It remains to show the interior majorization inequality \eqref{fullmaj1}.

Fix $r=1,\ldots,n$, and let $I_{0} = \{i:d_{i}<A_r\}$ and $I_{1} = \{i:d_{i}\geq A_r\}$. By the fact that $k_0=a-b$, we have
\[
\begin{aligned}
k_0- |\{i\in I: A_r \le d_i < A_n \}| 
& = \sum_{d_i < A_n} q_i - \sum_{d_i \ge A_n} (1-q_i) - |\{i\in I: A_r \le d_i < A_n \}|  
\\
& =\sum_{i\in I_0} q_i - \sum_{i\in I_1} (1-q_i).
\end{aligned}
\]
Thus, the required majorization \eqref{fullmaj1} is equivalent to
\begin{equation}\label{int4}
C(A_r)=\sum_{i\in I_0}\bigg( \sum_{j=1}^n 
 A_j p_i^{(j)} +
 Bq_i \bigg) \ge \sum_{j=1}^{r}A_{j} N_j + A_{r} \bigg(\sum_{i\in I_0} q_i - \sum_{i\in I_1} (1-q_i)  + \sum_{j=r+1}^{n}N_{j}\bigg).
\end{equation}
By \eqref{tt13}, we have for $j=1,\ldots,n$,
\[
N_j = \sum_{i\in I_0} p_i^{(j)} + \sum_{i\in I_1} p_i^{(j)}.
\]
Thus, \eqref{int4} can be rewritten as
\begin{equation}\label{int5}
\sum_{i\in I_0}\bigg( \sum_{j=r+1}^{n} (A_j -A_r) p_i^{(j)} + (B-A_r)q_i \bigg) 
\ge \sum_{i\in I_1} \bigg( 
 \sum_{j=1}^{r}A_{j} p_i^{(j)}+ 
 \sum_{j=r+1}^{n} A_r p_i^{(j)}
 + A_{r} (q_i-1)  \bigg).
\end{equation}
Since $\{A_j\}$ is an increasing sequence, the left hand side of \eqref{int5} is $\ge 0$. On the other hand, the right hand side of \eqref{int5} is $\le 0$ as it is dominated by
\[
\sum_{i\in I_1} \bigg( 
 \sum_{j=1}^{n}A_r p_i^{(j)}+ A_{r} q_i-A_r  \bigg) \le 0.
\]
In the last step we used \eqref{tt10}. This shows \eqref{int5}, which implies \eqref{int4}, thus proving \eqref{fullmaj1}. This completes the proof of Theorem \ref{int-nec}.
\end{proof}


\section{Sufficiency of interior majorization}

The goal of this section is to show the sufficiency in Theorem \ref{npt}. The sufficiency of condition (i), that is $C(B/2)+D(B/2)=\infty$, is a consequence of a result established by the second author, see \cite[Corollary 4.5]{jj}.

\begin{thm}\label{3intsuff} Let $\{A_j\}_{j=0}^{n+1}$ be an increasing sequence such that $A_0=0$ and $A_{n+1}=B$, $n\in \N$. Assume $\{d_{i}\}_{i\in I}$ is a sequence in $[0,B]$ such that for some (and hence all) $\alpha\in (0,B)$ we have
\[
C(\alpha) + D(\alpha)=\infty.
\]
Then there is a self-adjoint operator $E$ with $\sigma(E)=\{A_0,A_1,\ldots,A_{n+1}\}$ and diagonal $\{d_{i}\}_{i\in I}$.
\end{thm}

Next, we must demonstrate the sufficiency of condition (ii) of Theorem \ref{npt}. To achieve this we shall introduce an alternative variant of interior majorization. To distinguish between these two concepts we shall attach the name of Lebesgue to interior majorization that was defined in Definition \ref{mar}.

\begin{defn}\label{rim}
Suppose that $n\in \N$ and $\{A_{j}\}_{j=0}^{n+1}$ is an increasing sequence in $\R$ such that $A_0=0$ and $A_{n+1}=B$. 
Let $\{\lambda_{i}\}_{i\in\Z}$ be a nondecreasing sequence  which takes values in $\{A_{0},A_{1},\ldots,A_{n+1}\}$, each at least once. Let $\{d_{i}\}_{i\in\Z}$ be a nondecreasing sequence in $[0,B]$ such that $\sum_{i=-\infty}^0 d_i <\infty$.
We say that $\{d_i\}$ satisfies {\it Riemann interior majorization} by $\{A_j\}_{j=0}^{n+1}$ if there exists such a sequence $\{\lambda_{i}\}_{i\in\Z}$ as above, so that the following two hold:
\begin{align}
\label{rmaj1}\delta_{m}:=\sum_{i=-\infty}^{m}(d_{i}-\lambda_{i}) &\geq  0\qquad \text{for all }m\in\Z,
\\
\label{rmaj2}\lim_{m\to\infty}\delta_{m} & = 0.
\end{align}
\end{defn}

\begin{remark} 
 Since $\{\lambda_{i}\}$ is takes values $0$ and $B$ and is nondecreasing, there exists $k\in\Z$ and $N_1,\ldots,N_n\in \N$ such that 
\[
\lambda_{i}=\begin{cases}
0& i \le k, \\
A_r & k+ \sum_{j=1}^{r-1} N_j < i \le k+\sum_{j=1}^r N_j, \\
B & i > k+ \sum_{j=1}^n N_j.
\end{cases}
\]
\end{remark}

Our argument relies on the following two facts. The first result \cite[Theorem 5.2]{mbjj3} establishes the equivalence of two concepts of Lebesgue and Riemann interior majorization for nondecreasing sequences. Note that Theorem \ref{eqmajs} is concerned with numerical sequences without any mention to operators. The second result \cite[Theorem 5.3]{mbjj3} shows the existence of a self-adjoint operator with finite spectrum and prescribed diagonal under Riemann interior majorization.

\begin{thm}\label{eqmajs} Let $\{A_j\}_{j=0}^{n+1}$ be an increasing sequence in $\R$ with $A_{0}=0$ and $A_{n+1}=B$. Let $\{d_{i}\}_{i\in\Z}$ be a nondecreasing sequence in $[0,B]$.
Then, the sequence $\{d_{i}\}$ satisfies Lebesgue interior majorization by $\{A_j\}_{j=0}^{n+1}$ if and only if $\{d_{i}\}$ satisfies Riemann interior majorization by $\{A_j\}_{j=0}^{n+1}$.
\end{thm}

\begin{thm}\label{srim}
Let $\{A_j\}_{j=0}^{n+1}$ be an increasing sequence in $\R$ with $A_{0}=0$ and $A_{n+1}=B$. Let $\{d_{i}\}_{i\in\Z}$ be a nondecreasing sequence in $[0,B]$ that satisfies Riemann interior majorization by $\{A_j\}_{j=0}^{n+1}$.
Then, there is a self-adjoint operator $E$ with $\sigma(E) = \{A_0,\ldots,A_{n+1}\}$ and diagonal $\{d_{i}\}_{i\in\Z}$.
\end{thm}

By combining Theorems \ref{eqmajs} and \ref{srim} we can show the sufficiency of the Lebesgue interior majorization. In essence, we need to deal with sequences that satisfy Lebesgue interior majorization, but do not conform to more restrictive Riemann interior majorization. Theorem \ref{inthorn} shows the sufficiency part of Theorem \ref{npt}.

\begin{thm}\label{inthorn}
Let $\{A_j\}_{j=0}^{n+1}$ be an increasing sequence in $\R$ with $A_{0}=0$ and $A_{n+1}=B$. Let $\{d_{i}\}_{i\in I}$ be a sequence in $[0,B]$ that satisfies interior majorization by  $\{A_j\}_{j=0}^{n+1}$ and $\sum_{i\in I}d_{i} = \sum_{i\in I}(B-d_{i}) = \infty$.
Then, there is a self-adjoint operator $E$ with spectrum $\sigma(E) = \{A_0,\ldots, A_{n+1}\}$ and diagonal $\{d_{i}\}_{i\in I}$.
\end{thm}

\begin{proof} Set $J:=\{i\in I:d_{i}\in(0,B)\}$ and $J_{\lambda}:=\{i:d_{i}=\lambda\}$ for $\lambda=0,B$. Let ${\mathbf I}$ be the identity operator on a space of dimension $|J_{B}|$ and let $\bf{0}$ be the zero operator on a space of dimension $|J_{0}|$. Since $C(B/2)<\infty$ and $D(B/2)<\infty$, the only possible limit points of $\{d_{i}\}_{i\in J}$ are $0$ and $B$. The argument breaks into four cases depending on the number of limit points.

{\bf Case 1:} Assume both $0$ and $B$ are limit points of the sequence $\{d_{i}\}_{i\in J}$. This implies that there is a bijection $\pi:\Z\to J$ such that $\{d_{\pi(i)}\}_{i\in\Z}$ is in nondecreasing order. Since $\{d_{i}\}_{i\in J}$ still satisfies interior majorization, by Theorem \ref{eqmajs} the sequence $\{d_{\pi(i)}\}_{i\in \Z}$ satisfies Riemann interior majorization. By Theorem \ref{srim} there is a self-adjoint operator $E'$ with diagonal $\{d_{i}\}_{i\in J}$ and $\sigma(E')=\{A_0,\ldots,A_{n+1}\}$. The operator $E'\oplus B {\mathbf I}\oplus \bf{0}$ is as desired. This completes the proof of Case 1.

{\bf Case 2:} Assume $0$ is the only limit point of $\{d_{i}\}_{i\in J}$. Since $\sum_{i\in I}d_{i} = \infty$ we must have $|J_{B}|=\infty$. There is a bijection $\pi:\Z\to J\cup J_{B}$ such that $\{d_{\pi(i)}\}_{i\in\Z}$ is in nondecreasing order.  The sequence $\{d_{\pi(i)}\}_{i\in Z}$ satisfies interior majorization by $\{A_{j}\}_{j=0}^{n+1}$, and Theorem \ref{eqmajs} implies that it also satisfies Riemann interior majorization by $\{A_{j}\}_{j=0}^{n+1}$. By Theorem \ref{srim} there is a self-adjoint operator $E_{0}$ with diagonal $\{d_{\pi(i)}\}_{i\in \Z}$ and $\sigma(E_{0})=\{A_0,\ldots,A_{n+1}\}$. The operator $E=E_{0}\oplus\mathbf{0}$ has the same spectrum and diagonal $\{d_{i}\}$. This completes the proof of Case 2.

{\bf Case 3:} Assume $B$ is the only limit point of $\{d_{i}\}_{i\in J}$. The proof of this case follows by an obvious modification of Case 2.

{\bf Case 4:} Assume $\{d_{i}\}_{i\in J}$ has no limit points. This implies that $J$ is finite and since $\sum_{i\in I}d_{i} = \sum_{i\in I}(B-d_{i}) = \infty$ we also have $|J_{0}|=|J_{B}|=\infty$. There is a bijection $\pi:\Z\to I$ so that $\{d_{\pi(i)}\}_{i\in\Z}$ is nondecreasing. Theorem \ref{eqmajs} implies that $\{d_{\pi(i)}\}$ satisfies Riemann interior majorization by $\{A_{j}\}_{j=0}^{n+1}$. Theorem \ref{srim} implies that there is a self-adjoint operator $E$ with diagonal $\{d_{\pi(i)}\}$ and $\sigma(E) = \{A_0,\ldots,A_{n+1}\}$. This completes the proof of Case 4 and the theorem.
\end{proof}

The following example shows that there exist self-adjoint operators with the same spectrum and diagonal that are not unitarily equivalent. 

\begin{ex}\label{nu} Consider the sequence 
\[\{d_{i}\}_{i\in\Z} = \left\{\ldots,\frac{1}{8},\frac{1}{4},\frac{1}{2},1,\frac{3}{2},\frac{7}{4},\frac{15}{8},\ldots\right\}.\]
Set
\[E_{n} = \left[\begin{matrix} \frac{1}{2^{n}} & \sqrt{\frac{1}{2^{n}}\left(2-\frac{1}{2^{n}}\right)} \\ \sqrt{\frac{1}{2^{n}}\left(2-\frac{1}{2^{n}}\right)} & 2-\frac{1}{2^{n}}\end{matrix}\right].\]
For each $n\in\N$ we have $\sigma(E_{n})=\{0,2\}$, thus the operator $E=[1]\oplus E_{1}\oplus E_{2}\oplus\cdots$ has diagonal $\{d_{i}\}$ and $\sigma(E)=\{0,1,2\}$. Note that the multiplicity of the eigenvalue $1$ is $1$.

Alternatively, let $\{e_{i}\}_{i=1}^{\infty}$ be an orthonormal basis for a Hilbert space. Set $g=\sum_{i=1}^{\infty}2^{-i/2}e_{i}$, and define the projection $Pf = \langle f,g\rangle g$.
 The operator $P$ has diagonal $\{\frac{1}{2},\frac{1}{4},\frac{1}{8},\ldots\}$ in the basis $\{e_{i}\}$, and thus the operator $F=P\oplus[1]\oplus (2\mathbf{I}-P)$ has diagonal $\{d_{i}\}$. Since $P$ is a rank one projection we see that $\sigma(F) = \{0,1,2\}$, and the multiplicity of the eigenvalue $1$ is $3$. Thus, operators $E$ and $F$ share the same three point spectrum and diagonal, but they are not unitarily equivalent.
\end{ex}

We end the paper by a graphical example demonstrating Theorem \ref{npt}.

\begin{ex} Consider the sequence 
\[
\{d_{i}\}_{i\in\Z}=\bigg\{\ldots,\frac1{16},\frac18,\frac14,\frac12,\frac34,\frac78,\frac{15}{16},\ldots\bigg\}.
\]
By Kadison's Theorem \ref{Kadison} there does not exist a projection with diagonal $\{d_i\}$. However, in \cite{jj} it was shown that the set of possible $3$ point spectra of operators with the diagonal $\{d_i\}$
\[
\big\{A\in(0,1):\exists\,E\ \text{positive with diagonal }\{d_{i}\}_{i\in\N}\text{ and }\sigma(E)=\{0,A,1\}\big\}\]
consists of exactly $7$ points $\{\frac18,\frac16,\frac14,\frac12,\frac34,\frac56,\frac78\}$.
With the help of {\it Mathematica} and the characterization from Theorem \ref{npt} we can find the corresponding set of possible $4$ point spectra of operators. The following figure shows the set 
\[
\big\{(A_{1},A_{2})\in(0,1)^2:\exists\,E\geq 0\text{ with } \sigma(E)=\{0,A_{1},A_{2},1\}\text{ and diagonal } \{d_{i}\}\big\}.
\]
For the study of properties of such sets we refer to \cite{mbjj4}.
\end{ex}
\centerline{\includegraphics[width=2.8in]{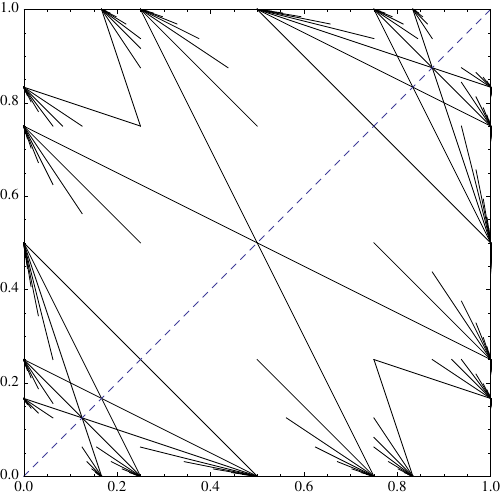}}


\begin{thebibliography}{99}




\bibitem{a}
W. Arveson, 
{\it Diagonals of normal operators with finite spectrum},
Proc. Natl. Acad. Sci. USA {\bf 104} (2007), 1152--1158.

\bibitem{ak}
W. Arveson, R. Kadison, 
{\it Diagonals of self-adjoint operators}, Operator theory, operator algebras, and applications, 247--263, Contemp. Math., {\bf 414}, Amer. Math. Soc., Providence, RI, 2006.



\bibitem{mbjj3}
M.~Bownik, J.~Jasper,
{\it 
The Schur-Horn Theorem for operators with finite spectrum},
Tran. Amer. Math. Soc. (to appear).

\bibitem{mbjj4}
M.~Bownik, J.~Jasper,
{\it Spectra of frame operators with prescribed frame norms},
Proceedings of the 9th International Conference on  Harmonic Analysis and Partial Differential Equations, Contemp. Math. {\bf 612} (2014), 65--79.

\bibitem{gm}
I. C. Gohberg, A. S. Markus, 
{\it Some relations between eigenvalues and matrix elements of linear operators},  Mat. Sb. (N.S.) {\bf 64} (1964), 481--496.



\bibitem{horn}
A. Horn,
{\it  Doubly stochastic matrices and the diagonal of a rotation matrix},
Amer. J. Math. {\bf 76} (1954), 620--630.

\bibitem{jj}
J. Jasper,
{\it The Schur-Horn theorem for operators with three point spectrum},
J. Funct. Anal. {\bf 265} (2013), 1494--1521.

\bibitem{k1}
R. Kadison, 
{\it The Pythagorean theorem. I. The finite case},
Proc. Natl. Acad. Sci. USA {\bf 99} (2002), 4178--4184.

\bibitem{k2}
R. Kadison, 
{\it The Pythagorean theorem. II. The infinite discrete case},
Proc. Natl. Acad. Sci. USA {\bf 99} (2002), 5217--5222.


\bibitem{kw0}
V. Kaftal, G. Weiss,
{\it A survey on the interplay between arithmetic mean ideals, traces, lattices of operator ideals, and an infinite Schur-Horn majorization theorem}, Hot topics in operator theory, 101--135, Theta Ser. Adv. Math., 9, Theta, Bucharest, 2008.

\bibitem{kw}
V. Kaftal, G. Weiss,
{\it An infinite dimensional Schur-Horn theorem and majorization theory},  J. Funct. Anal. {\bf 259} (2010), 3115--3162.


\bibitem{neu}
A. Neumann, 
{\it An infinite-dimensional version of the Schur-Horn convexity theorem},
J. Funct. Anal. {\bf 161} (1999), 418--451.

\bibitem{schur}
I. Schur,
{\it \"Uber eine Klasse von Mittelbildungen mit Anwendungen auf die Determinantentheorie}, 
Sitzungsber. Berl. Math. Ges. {\bf 22} (1923), 9--20.




\end{thebibliography}
\end{document}